\documentclass[12pt]{amsart}
\usepackage{setspace}
\usepackage{geometry}
\usepackage{todonotes}
\usepackage{amsthm}
\usepackage{amsfonts, amssymb, amscd}
\numberwithin{equation}{section}
\usepackage[symbol]{footmisc}

\usepackage{bm}
\usepackage{bbm}
\usepackage{verbatim}
\usepackage{mathrsfs}
\usepackage{graphicx}
\usepackage{tikz-cd}
\usepackage{subcaption}
\usepackage{listings}
\usepackage{subfiles}
\usepackage[toc,page]{appendix}
\usepackage{mathtools}
\usepackage{comment}
\usepackage{enumerate}
\usepackage{enumitem}
\usepackage[all]{xy}
\usepackage{graphicx}
\graphicspath{{images/}}
\usepackage{appendix}
\usepackage[colorlinks=true, citecolor=green, linkcolor=blue, filecolor=magenta, urlcolor=red,backref=page]{hyperref}

\usepackage{pgfplots}
\pgfplotsset{compat=1.17}

\newcommand{\ov}[1]{\overline{#1}}
\newcommand{\bQ}{\mathbb{Q}}
\newcommand{\bR}{\mathbb{R}}

\newcommand{\bN}{\mathbb{N}}
\newcommand{\bP}{\mathbb{P}}

\newcommand{\cA}{\mathcal{A}}

\newcommand{\cE}{\mathcal{E}}
\newcommand{\cF}{\mathcal{F}}
\newcommand{\cG}{\mathcal{G}}
\newcommand{\cH}{\mathcal{H}}

\newcommand{\cK}{\mathcal{K}}
\newcommand{\cL}{\mathcal{L}}

\newcommand{\cN}{\mathcal{N}}
\newcommand{\cO}{\mathcal{O}}

\newcommand{\cQ}{\mathcal{Q}}

\newcommand{\Pic}{\mathrm{Pic}}

\newcommand{\Supp}{\mathrm{Supp}}

\newcommand{\vol}{\operatorname{vol}}

\newcommand{\lct}{\operatorname{lct}}

\newcommand{\Spec}{\operatorname{Spec}}

\newcommand{\mult}{\operatorname{mult}}
\theoremstyle{plain} % 'this is the initial setting and can be omitted here'
\newtheorem{thm}{Theorem}[section] % number like 3.1, 3.2, 3.3, etc.
\newtheorem{lemma}[thm]{Lemma}
\newtheorem{prop}[thm]{Proposition}

\theoremstyle{definition} % 'here we change the style'
\newtheorem{defn}[thm]{Definition} % numbered with thm

\newtheorem{rem}[thm]{Remark}
\newtheorem{claim}[thm]{Claim}
\newtheorem{setup}[thm]{Set-up}
\begin{document}
	
	\title{Quasi-Projective Moduli for Polarized klt Good Minimal Models}
	\author{Xiaowei Jiang}
	\email{jxw20@mails.tsinghua.edu.cn}
	\address{Yau Mathematical Sciences Center, Tsinghua University, Beijing, China}
	\date{\today}
	
	\thanks{}

	%%%%%%%%%%%%%%%%%%%%%%%%%%%%%%%%%%%%%%%%%%%%%%%%%%%%%%%%%%Abstract
	\begin{abstract}
We prove the weak positivity of direct images for locally stable families of klt good minimal models over reduced quasi-projective bases using Gabber's Extension Theorem. As an application, we apply Viehweg's ampleness criterion to show that the normalization of the moduli space of polarized klt good minimal models of arbitrary Kodaira dimension, constructed in \cite{jiangBoundednessKltGood2023}, is quasi-projective.
	\end{abstract}
	
	\maketitle
	\tableofcontents
	%%%%%%%%%%%%%%%%%%%%%%%%%%%%%%%%%%%%%%%%%%%	 Introduction
	\section{Introduction}
	Throughout this paper, we work over an algebraically closed field $k$ of characteristic zero. We let $\mathbb{N}=\{1,2,\dots\}$ denote the set of positive integers.
	\vspace*{20pt}

The central problem in birational geometry is the classification of algebraic varieties. From the perspective of the minimal model program, canonically polarized varieties, Fano varieties, Calabi--Yau varieties, and their iterated fibrations play a central role in this classification. Closely related to this problem is the construction of moduli spaces of such higher-dimensional algebraic varieties. Over the past decades, significant progress has been made in this direction.

For canonically polarized varieties and Fano varieties, the projective moduli theory has been well developed (cf. \cite{kollarFamiliesVarietiesGeneral2023,xuKstabilityFanoVarieties2025}). This theory is built on the development of the higher-dimensional minimal model program (cf. \cite{birkarExistenceMinimalModels2010,birkarExistenceLogCanonical2012,haconExistenceLogCanonical2013}), on boundedness results (cf. \cite{haconBirationalAutomorphismsVarieties2013,haconACCLogCanonical2014,haconBoundednessModuliVarieties2018,birkarAntipluricanonicalSystemsFano2019,birkarSingularitiesLinearSystems2021}), and on positivity results (cf. \cite{fujinoSemipositivityTheoremsModuli2018,kovacsProjectivityModuliSpace2017,codogniPositivityCMLine2021,xuPositivityCMLine2020}).

However, for polarized Calabi--Yau varieties and more general polarized good minimal models, boundedness problems create substantial difficulties. Let us explain this more precisely.
If one restricts to smooth varieties, then one may take a line bundle as the polarization. This setting was studied in Viehweg's series of works \cite{viehwegWeakPositivityStability1989,viehwegWeakPositivityStability1990,viehwegWeakPositivityStability1990a,viehwegQuasiprojectiveQuotientsCompact1991,esnaultAmpleSheavesModuli1991,viehwegQuasiprojectiveModuliPolarized1995,viehwegCompactificationsSmoothFamilies2010}, where he proved that the moduli space of smooth polarized good minimal models is quasi-projective.
Since such moduli spaces are not proper, it is necessary to allow singular degenerations. In this setting, the limiting polarization is no longer Cartier, but only a $\mathbb{Q}$-Cartier Weil divisor. Restricting to the klt case, one can nevertheless show that the Cartier index of the limiting polarization is bounded \cite{birkarGeometryPolarisedVarieties2023,jiangBoundednessKltGood2023}.
The situation becomes considerably more complicated for varieties with slc singularities. In general, the Cartier index of the polarization of slc good minimal models is not bounded; see \cite{haconFailureBoundednessGeneralised2025} for surface examples with various Kodaira dimensions. As a consequence, constructing compact moduli spaces of good minimal models becomes significantly more difficult. We refer to \cite{birkarModuliAlgebraicVarieties2022,birkarGeometryPolarisedVarieties2023,ascherModuliBoundaryPolarized2023,blumGoodModuliSpaces2024,bakkerBailyBorelCompactificationsPeriod2025a} for recent developments.

The moduli space of polarized klt good minimal models was constructed in \cite{jiangBoundednessKltGood2023}, where the corresponding moduli functor $\mathfrak{G}_{klt}(d,u,\sigma)$ (see Definition~\ref{def:moduli functor}) is shown to be a separated Deligne--Mumford stack of finite type admitting a coarse moduli space $G_{klt}(d,u,\sigma)$ as a separated algebraic space.
In this paper, we study $G_{klt}(d,u,\sigma)$ further. Although it is non-proper in general, we prove that its normalization is quasi-projective.  We also refer to \cite{hattoriPositivityCMLine2026} for related results on the moduli space of klt good minimal models of Kodaira dimension one.

\begin{thm}[= Theorem \ref{thm:quasi moduli'}]\label{thm:quasi moduli}
 Let $\delta:\widetilde{G}_{klt}(d,u,\sigma)\to G_{klt}(d,u,\sigma)$ be the normalization of $G_{klt}(d,u,\sigma)$, then $\widetilde{G}_{klt}(d,u,\sigma)$ is a quasi-projective scheme. 
Moreover, if the non-normal locus of $G_{klt}(d,u,\sigma)$  is proper, then  $G_{klt}(d,u,\sigma)$ is a quasi-projective scheme.
\end{thm}

The general strategy for proving the projectivity of proper algebraic moduli spaces is based on the so-called ampleness lemma \cite{kollarProjectivityCompleteModuli1990} and its refined versions \cite{kovacsProjectivityModuliSpace2017,xuPositivityCMLine2020}. For non-proper moduli spaces, GIT methods are used for the moduli of smooth curves \cite{mumfordGeometricInvariantTheory1994}, smooth surfaces of general type \cite{giesekerGlobalModuliSurfaces1977}, and smooth polarized good minimal models \cite{viehwegWeakPositivityStability1989,viehwegWeakPositivityStability1990,viehwegWeakPositivityStability1990a}. In \cite[\S4]{viehwegWeakPositivityStability1989}, Viehweg also proved an ampleness criterion in the non-proper setting that avoids GIT; see \S\ref{sec:Viehweg's ampleness criterion} for details. The main task in this approach is to establish a non-proper analogue of nefness in the proper setting, called weak positivity; we refer the reader to \S\ref{sec:wp} for the relevant background.

We prove the following weak positivity result for locally stable families of polarized klt good minimal models over reduced quasi-projective bases, which is crucial for Theorem~\ref{thm:quasi moduli}. It can be viewed as a klt pairs analogue of \cite[Theorem 6.24]{viehwegQuasiprojectiveModuliPolarized1995}.

\begin{thm}[= Theorem \ref{thm:Viehweg package'}]\label{thm:Viehweg package}
Let $f:(X,B)\to S$ be a locally stable family of klt pairs over a reduced quasi-projective scheme $S$ such that $K_X+B$ is $f$-semiample. 
Let $L$ be an $f$-semiample line bundle on $X$.
Assume that:
\begin{itemize}
 % \item    $c(K_{X/S}+B)$ is Cartier for some natural number $c\geq 2$,
    \item there exists $l\in \bN$ such that 
$\lct(X_s,B_s,|L_s|)>\frac{1}{l}$ for all $s\in S$, and
\item the sheaf $f_*L$ is locally free of rank $r>0$ and compatible with arbitrary base change.
\end{itemize}

Then  for every natural number $q\geq 2$ such that $q(K_{X/S}+B)$ is Cartier,   we have:
\begin{enumerate}
    \item 
    % {\normalfont (\textbf{Base change and locally freeness})} 
    the sheaf
    $R^if_*\bigl(\cO_X(ql(K_{X/S}+B))\otimes L^q\bigr)$
    is locally free 
    %of rank $r'$ 
    and compatible with arbitrary base change  for all $i\geq0$, and

    \item
    % {\normalfont (\textbf{Weak positivity})} 
     the sheaf
    $$
    \left(\bigotimes^{r} f_*\bigl(\cO_X(ql(K_{X/S}+B))\otimes L^q\bigr)\right)\otimes\det(f_*L)^{-q}
    $$
    is weakly positive over $S$.

%   \item {\normalfont (\textbf{Weak stability})}  For $p,q \in \mathbb{N}$ divisible by $c$,  if $r' \neq 0$, then there exists a positive rational number $\delta$ such that
% \[
% \begin{aligned}
% &\left(\bigotimes^{r} f_*\bigl(\cO_X(pl(K_{X/S}+B))\otimes L^p\bigr)\right)\otimes\det(f_*L)^{-p} \\
% &\succeq\ 
% \delta \cdot \det\!\left(f_*\bigl(\cO_X(ql(K_{X/S}+B))\otimes L^q\bigr)\right)^{r}
% \otimes \det(f_*L)^{-q r'}.
% \end{aligned}
% \]
\end{enumerate}
\end{thm}
By Viehweg's fiber product trick and a careful analysis of the base change property for direct image sheaves, we reduce Theorem~\ref{thm:Viehweg package} to the following weak positivity result for locally stable families of klt good minimal models without polarizations. This result may be regarded as a klt pairs counterpart of \cite[Theorem 6.22]{viehwegQuasiprojectiveModuliPolarized1995}.
    
\begin{thm}[= Theorem \ref{thm: wp of direct image  of family of gmm'}]\label{thm: wp of direct image  of family of gmm}
Let $f:(X,B)\to S$ be a locally stable family of klt pairs over a reduced quasi-projective scheme $S$ such that $K_X+B$ is $f$-semiample. 
Then for every natural number $q\geq 2$ such that $q(K_{X/S}+B)$ is Cartier, 
\begin{enumerate}
    \item the sheaf $R^if_*\cO_X\bigl(q(K_{X/S}+B)\bigr)$ is locally free  and compatible with arbitrary base change for all $i\geq 0$, and
    \item the sheaf $f_*\cO_X\bigl(q(K_{X/S}+B)\bigr)$ is weakly positive over $S$.
\end{enumerate}
\end{thm}

\noindent \textbf{Sketch of the proof of Theorem~\ref{thm: wp of direct image  of family of gmm}.}
If the base $S$ is projective and normal, a more general form of Theorem~\ref{thm: wp of direct image  of family of gmm} is known; see Theorem~\ref{thm:semipositivity}. Therefore, applying Weak Semistable Reduction \cite{abramovichWeakSemistableReduction2000} and the MMP for lc pairs \cite{birkarExistenceLogCanonical2012,haconExistenceLogCanonical2013}, after a generically finite base change, we can extend the direct image sheaves to nef locally free sheaves on a proper smooth base. However, weak positivity is not functorial on non-proper or non-normal bases, so it cannot be directly descended to the original base. A key input is Gabber's Extension Theorem (see \S\ref{sec:gabber}), which, under certain technical conditions, allows one to extend certain locally free sheaves to a compactification of a finite cover of $S$. To verify these conditions, we construct such a finite cover of $S$ with splitting trace map inductively using a stratification of $S$, and then descend weak positivity to $S$.

\vspace*{15pt}
\noindent	\textbf{Acknowledgement.}
The author would like to thank Professor Florin Ambro for helpful discussions and comments. He is also very grateful to his advisor Professor Caucher Birkar for constant support and encouragement.

	%%%%%%%%%%%%%%%%%%%%%%%%%%%%%%%%%%%%%%%%%%%Preliminaries
	\section{Preliminaries}

\subsection{Pairs and singularities}

\begin{defn}[Pairs and singularities]

A \textit{pair} $(X,B)$ consists of a quasi-projective variety $X$ and an $\bR$-divisor $B$ on $X$ such that $K_X+B$ is $\bR$-Cartier. 
If $X$ is normal,
let $D$ be a prime divisor over $X$. Let $\pi \colon X' \to X$ be a log resolution of $(X,B)$ such that $D$ is a prime divisor on $X'$. We can write
\[
K_{X'} + B' = \pi^*(K_X + B).
\]
The \textit{log discrepancy} of $D$ with respect to $(X,B)$ is defined by
\[
a(D,X,B) := 1 - \mult_D B'.
\]
We say that $(X,B)$ is \textit{klt} (resp. \textit{lc}, \textit{$\epsilon$-lc}) if $a(D,X,B) > 0$ (resp. $a(D,X,B) \ge 0$, $a(D,X,B) \ge \epsilon$) for every prime divisor $D$ over $X$.
\end{defn}

\begin{defn}[Demi-normal]
Recall that, by Serre's criterion, a scheme $X$ is normal if and only if it is $S_2$ and regular at all codimension one points.
As a weakening of normality, a scheme is called \emph{demi-normal} if it is $S_2$ and its codimension one points are either regular points or nodes.

Let $\pi:\ov{X}\to X$ denote the normalization and $D\subset X$ the divisor obtained as the closure of the nodes of $X$. Set $\ov{D}:=\pi^{-1}(D)$ with reduced structure. Then $D,\ov{D}$ are the \textit{conductors} of $\pi$, and the induced map $\ov{D}\to D$ has degree $2$ over the generic points. 
\end{defn}
\begin{defn}[Slc singularities]
Let $X$ be a demi-normal quasi-projective variety with normalization $\pi:\ov{X}\to X$ and with conductors $D\subset X$ and $\ov{D}\subset\ov{X}$. Let $B$ be an effective $\mathbb{R}$-divisor whose support does not contain any irreducible component of $D$, and $\ov{B}$ the divisorial part of $\pi^{-1}(B)$. The pair $(X,B)$ is called \textit{semi-log-canonical}, or \textit{slc}, if
\begin{enumerate}
\item $K_X+B$ is $\mathbb{R}$-Cartier, and
\item $(\ov{X},\ov{D}+\ov{B})$ is lc.
\end{enumerate}
\end{defn}

\begin{defn}[Minimal models]

 Suppose that $f : X \to Z$ and $f ^m : X ^m\to Z$ are projective morphisms,  $\phi:X\dashrightarrow X^m$ is a birational map over $Z$ which does not extract any divisor, and $(X, B)$ and $(X^ m , B^m )$ are lc pairs, where $B^m=\phi_*B$. If $a(E, X, B) > a(E, X ^m , B^ m )$ (resp. $a(E, X, B) \geq a(E, X ^m , B^ m )$) for all prime $\phi$-exceptional divisors $E \subset X$, $X^ m$ is $\bQ$-factorial, and $K_{X^m}+B^m$ is nef over $Z$, then we say that $\phi:X\dashrightarrow X^m$  is a \textit{minimal model} (resp. \textit{weak log canonical model})  of $(X,B)$ over $Z$. 
 
  A minimal model (resp. weak log canonical model) $\phi:X\dashrightarrow X^m$  of $(X,  B)$ over $Z$ is called a \textit{good minimal model} (resp. \textit{semiample model}) if $K_{X^m}+B^m$ is semiample over $Z$. 
  \end{defn}

\begin{defn}[Lc threshold of a line bundle]
   Let $(X, B)$ be an lc pair.  
The \textit{lc threshold} of an effective $\mathbb{Q}$-Cartier $\mathbb{Q}$-divisor $D$ with respect to $(X, B)$ is defined as
\[
\operatorname{lct}(X,B,D)
:=
\sup \left\{\, t\in\mathbb{R}\mid (X,B+tD)\text{ is lc} \,\right\}.
\]

Now let $L$ be a line bundle.  
We define the \textit{lc threshold} of $L$ with respect to $(X,B)$ to be the infimum of the log canonical thresholds of effective divisors in the complete linear system $|L|$:
\[
\operatorname{lct}(X,B,L)
:=
\inf \left\{\, \operatorname{lct}(X,B,D)\mid D\in |L| \,\right\}.
\]
    
\end{defn}

   \subsection{Weakly positive sheaves}\label{sec:wp} 
  In this subsection, we recall weakly positive sheaves and their properties from \cite[\S2]{viehwegQuasiprojectiveModuliPolarized1995}.
\begin{defn}[{Nef locally free sheaves, \cite[Proposition 2.9]{viehwegQuasiprojectiveModuliPolarized1995}}]\label{def:nef}
Let $\cG$ be a locally free sheaf on a proper scheme $S$. We say that $\cG$ is \textit{nef} if the following equivalent conditions hold:
\begin{enumerate}
\item For every non-singular projective curve $C$ and every morphism $f : C \to S$, every invertible quotient sheaf $\cN$ of $f^*\cG$ satisfies
$\deg(\cN) \ge 0$.

\item Let $\pi : \mathbb{P}(\cG) \to S$ be the projective bundle associated to $\cG$, and let $\mathcal{O}_{\mathbb{P}(\cG)}(1)$ be the tautological line bundle. Then $\mathcal{O}_{\mathbb{P}(\cG)}(1)$ is nef.

\item For an ample invertible sheaf $\cH$ on $S$, and for all $\alpha > 0$, there exists $\beta > 0$ such that
$S^{\alpha \cdot \beta}(\cG) \otimes \cH^{\beta}$
is generated by global sections.
\end{enumerate}
\end{defn}

\begin{lemma}[{\cite[Lemma 2.8]{viehwegQuasiprojectiveModuliPolarized1995}}]\label{lem:nef functorial}
Let $\cG$ be a locally free sheaf of rank $r$ on a reduced proper scheme $S$, and let $
\tau: S' \to S$ be a proper morphism.
\begin{enumerate}
\item If $\cG$ is nef, then $\tau^*\cG$ is nef.

\item If $\tau$ is surjective and $\tau^*\cG$ is nef, then $\cG$ is nef.
\end{enumerate}
\end{lemma}

If $S$ is quasi-projective, the numerical characterization in Definition~\ref{def:nef}.(1) and the functorial property in Lemma~\ref{lem:nef functorial} are no longer available. We will instead use Definition~\ref{def:nef}.(3) to define weak positivity in the quasi-projective case.

\begin{defn}
Let $S$ be a quasi-projective scheme, let $S_0 \subset S$ be an open subscheme, and let $\cG$ be a coherent sheaf on $S$. We say that $\cG$ is \textit{globally generated over $S_0$} if the natural map
\[
H^0(S,\cG) \otimes_k \mathcal{O}_S \to \cG
\]
is surjective over $S_0$.
\end{defn}

\begin{defn}[Weakly positive sheaves]
Let $S$ be a quasi-projective reduced scheme, let $S_0 \subseteq S$ be an open dense subscheme, and let $\cG$ be a locally free sheaf on $S$ of finite constant rank. Then $\cG$ is called \textit{weakly positive over $S_0$} if for an ample invertible sheaf $\cH$ on $S$ and for a given number $\alpha > 0$, there exists some $\beta > 0$ such that $S^{\alpha \cdot \beta}(\cG) \otimes \cH^{\beta}$ is globally generated over $S_0$.
\end{defn}

Assume that $S$ is projective. Then the sheaf $\cG$ is weakly positive over $S$ if and only if it is nef.

\begin{rem} There is another weaker definition of weak positivity in the literature; see for example \cite{fujinoNotesWeakPositivity2017}\cite[Definition 4.7]{kovacsProjectivityModuliSpace2017}. It only requires that $S^{\alpha \cdot \beta}(\cG) \otimes \cH^{\beta}$ is generically globally generated, without specifying the locus of global generation. This notion is easier to verify than Viehweg's original notion of weak positivity, and it is useful for proving the bigness of certain direct image sheaves in the study of the Iitaka conjecture.
\end{rem}

The functorial properties of weakly positive sheaves with respect to finite morphisms are more subtle.
\begin{lemma}[{\cite[Lemma 2.15]{viehwegQuasiprojectiveModuliPolarized1995}}]
\label{lem:wp functorial}
Let $S$ be a quasi-projective reduced scheme, let $S_0 \subseteq S$ be an open dense subscheme, and let $\cG$ be a locally free sheaf on $S$ of finite constant rank. Let $\tau : S' \to S$ be a morphism of reduced quasi-projective schemes.

\begin{enumerate}
\item If $\cG$ is weakly positive over $S_0$, then $\tau^*\cG$ is weakly positive over $S_0':= \tau^{-1}(S_0)$.

\item If $\tau : S' \to S$ is surjective such that $\tau^*\cG$ is weakly positive over $S_0' $, the restriction $\tau_0 := \tau|_{S_0'}$ is finite, and the trace map splits the inclusion
$\mathcal{O}_{S_0} \to \tau_{0*}\mathcal{O}_{S_0'}$,
then $\cG$ is weakly positive over $S_0$.
\end{enumerate}
\end{lemma}

%   \begin{defn}
%        Let $\cF$ be a locally free sheaf and let $\cA$ be an invertible sheaf, both on a quasi-projective reduced scheme $S$.
% For $\alpha,\beta \in \mathbb{N}$, we write
% $\cF\succeq\frac{\beta}{\alpha}\cA$
% if
% $S^\alpha(\cF) \otimes \cA^{-\beta}$
% is weakly positive over $S$.
%    \end{defn}

\subsection{Gabber's Extension Theorem}\label{sec:gabber}
In this subsection, we recall Gabber's extension result from \cite[\S5]{viehwegQuasiprojectiveModuliPolarized1995} (see also \cite[\S12]{viehwegCompactificationsSmoothFamilies2010}).

\begin{lemma}\label{lem:trace}
Let $\mathbb{P}$ be an irreducible normal projective variety, let $\Psi:\mathbb{P}'\to \mathbb{P}$ be a finite normal covering, and let $\overline{S}\subset \mathbb{P}$ be a closed subvariety. Then $\ov{\phi}:\overline{T}:=\Psi^{-1}(\overline{S})\to \overline{S}$ admits a splitting trace map.
\end{lemma}

\begin{proof}
Since $\mathbb{P}'$ is normal, $\mathcal{O}_{\mathbb{P}}$ is a direct summand of $\Psi_*\mathcal{O}_{\mathbb{P}'}$. Then $\ov{\phi}:\overline{T}\to \overline{S}$ is automatically finite, and the trace map gives a surjection
\[
\Psi_*\mathcal{O}_{\mathbb{P}'}\to\mathcal{O}_{\mathbb{P}}\to \mathcal{O}_{\overline{S}}.
\]
Moreover, the composed map factors through
$\ov{\phi}_*\mathcal{O}_{\overline{T}}\to \mathcal{O}_{\overline{S}}$.
\end{proof}

\begin{lemma}[{\cite[Lemma 12.3]{viehwegCompactificationsSmoothFamilies2010}, \cite[Claim 5.8]{viehwegQuasiprojectiveModuliPolarized1995}}]\label{lem: dominant cover}
   If $\lambda : W \to \mathbb{P}^M$ is a morphism such that each component of $W$ is generically finite over its image, then there exists a finite normal Galois cover $\Psi : \mathbb{P}' \to \mathbb{P}^M$, a scheme $\widetilde{W}$ birational to $W$, and a subscheme $T:= \Psi^{-1}(\lambda(W))$ of $\mathbb{P}'$ such that $\Psi|_{T}$ factors through $\widetilde{W}$.
\end{lemma}

\begin{setup}\label{setup}
    
Let $\mathbb{P}$ be a normal projective scheme, 
$\overline{S}\subset \mathbb{P}$ a closed reduced subscheme,
and $S \subset \overline{S}$ an open dense subscheme. 
Let $\Psi: \mathbb{P}' \to \mathbb{P}$ be a finite covering with $\mathbb{P}'$ normal. 
Write $\overline{T} = \Psi^{-1}(\overline{S})$, $T= \Psi^{-1}(S)$, and denote by
$\overline{\phi}= \Psi|_{\overline{T}}$ and $\phi = \Psi|_T$.
Consider a modification $\xi: T' \to T$ with $T'$ nonsingular, and let $\overline{T'}$ be a projective smooth variety containing $T'$ as an open dense subscheme.

Let $\mathcal{F}$ and $\mathcal{G}_{\ov{T'}}$ be locally free sheaves on $S$ and $\overline{T'}$ respectively, and set $\mathcal{F}_{T} := \phi^*\mathcal{F}$. We assume the following:

\begin{enumerate}
    \item $\xi^*\mathcal{F}_{T} = \mathcal{G}_{\ov{T'}}|_{T'}$.
\item $\mathcal{G}_{\ov{T'}}$ is nef.
  \item For each morphism $\overline{\eta}: \overline{C} \to \mathbb{P}'$ from a nonsingular projective curve $\overline{C}$, with $C = \overline{\eta}^{-1}(T)$ dense in $\overline{C}$, the sheaf $\cF_C:=(\overline{\eta}|_{C})^*\mathcal{F}_{T}$ extends to a locally free sheaf $\mathcal{G}_{\overline{C}}$ satisfying:

\begin{enumerate}
   \item[(a)] If $\overline{\eta}': \overline{C'} \to \mathbb{P}'$ factors through $\overline{\gamma}: \overline{C'}\to\overline{C} $, then $\mathcal{G}_{\overline{C'}} = \overline{\gamma}^*\mathcal{G}_{\overline{C}}$.

\item[(b)] If $\overline{\eta}: \overline{C} \to \mathbb{P}'$ lifts to a morphism $\overline{\chi}: \overline{C} \to \overline{T'}$, then $\mathcal{G}_{\overline{C}} = \overline{\chi}^*\mathcal{G}_{\overline{T'}}$.
\end{enumerate}
\end{enumerate}
\end{setup}

\begin{thm}[Gabber's Extension Theorem]\label{thm:Gabber extension}

In Set-up~\ref{setup}, after blowing up $\mathbb{P}$ along centers disjoint from $S$ and replacing $\mathbb{P}'$ by the normalization of $\mathbb{P}$ in its function field, the following hold:
\begin{enumerate}
    \item There exists an extension of $\mathcal{F}_{T}$ to a locally free sheaf $\mathcal{F}_{\overline{T}}$ on $\overline{T} = \Psi^{-1}(\overline{S})$ such that for all commutative diagrams
\[
\begin{tikzcd}[row sep=1.5em, column sep=3.5em]
T' \arrow[r, "\subset"] \arrow[d, "\xi"']
& \overline{T'} 
& \Lambda \arrow[l, "\rho"'] \arrow[ld, "\psi"'] \\
T \arrow[r, "\subset"] \arrow[d, "\phi"']
& \overline{T} \arrow[r, "\subset"] \arrow[d, "\overline{\phi}"']
& \mathbb{P}' \arrow[d, "\Psi"] \\
S \arrow[r, "\subset"]
& \overline{S} \arrow[r, "\subset"]
& \mathbb{P}
\end{tikzcd}
\]
with $\Lambda$ proper and nonsingular, and with $\psi$ and  $\rho$ are two birational morphisms, one has
$\psi^*\mathcal{F}_{\overline{T}} = \rho^*\mathcal{G}_{\ov{T'}}$.

    \item $\mathcal{F}_{\overline{T}}$ is nef.

    \item $\mathcal{F}$ is weakly positive over $S$.
\end{enumerate}
\end{thm}

\begin{proof}
    
 (1). This is essentially the content of \cite[Theorem 5.1 and Variant 5.6]{viehwegQuasiprojectiveModuliPolarized1995}, where a compactification $\overline{T}$ of $T$ and a locally free sheaf $\mathcal{F}_{\overline{T}}$ are constructed. Replacing $\overline{T}$ by a further modification with centers outside $T$ if necessary, \cite[Lemma 12.4]{viehwegCompactificationsSmoothFamilies2010} implies that $\overline{T}$ can be embedded into a modification of $\mathbb{P}'$ which is finite over a modification of $\mathbb{P}$.

 (2). Since $\mathcal{G}_{\ov{T'}}$ is nef, it follows from (1) and Lemma~\ref{lem:nef functorial} that $\mathcal{F}_{\overline{T}}$ is nef. In particular, $\mathcal{F}_{T}=\phi^*\mathcal{F}$ is weakly positive over $T$.

(3). By Lemma~\ref{lem:trace}, the morphism $\phi:T\to S$ admits a splitting trace map. Hence, by Lemma~\ref{lem:wp functorial}.(2), $\mathcal{F}$ is weakly positive over $S$.
\end{proof}

	\subsection{Locally stable family}
In this subsection,
we  recall foundations from \cite{kollarFamiliesVarietiesGeneral2023}.

\begin{defn}[Mostly flat family of coherent sheaves]\label{def:most flat}
Let $f: X \to S$ be a morphism and $\cF$ a coherent sheaf on $X$.
We say that $\cF$ is \textit{mostly flat} over $S$ if there exists a dense open subset
$j: U \hookrightarrow X$ such that
\begin{enumerate}
\item $\cF|_U$ is flat over $S$, and
\item 
$\operatorname{codim}_{\operatorname{Supp} F_s}
\bigl(\operatorname{Supp} \cF_s \setminus U\bigr)
\ge 2$, for every $s \in S$.

\end{enumerate}
We usually set $Z := X \setminus U$.
\end{defn}

\begin{defn}[Relative hull and hull pull-back]
With $j: U \hookrightarrow X$ as in Definition~\ref{def:most flat}, let $\cF$ be a mostly flat family of coherent sheaves. Assume that $\cF|_U$ has $S_2$ fibers. We define
$$
\cF^H := j_*(\cF|_U)
$$
to be the \textit{relative hull} of $F$.

Let $q: T \to S$ be a morphism and set $X_T := X \times_S T$ with projection $q_X: X_T \to X$. Let $U_T := q_X^{-1}(U)$.
Then $\cF_T:= q_X^* \cF$ has $S_2$ fibers over $U_T$.
We define the \textit{hull pull-back} $(\cF_T)^H$ of $F$ to be the relative hull of $\cF_T$, that is,
\[
\cF_T^H := j_{T*}\bigl(\cF_T|_{U_T}\bigr),
\]
where $j_T: U_T \hookrightarrow X_T$ is the natural open immersion.
\end{defn}

\begin{defn}[Flat family of divisorial sheaves]\label{def:Flat family of divisorial sheaves}
A coherent sheaf $\cL$ on a scheme $X$ is called a \emph{divisorial sheaf}
if $\cL$ is $S_2$ and there exists a closed subset $Z \subset X$ of codimension $\ge 2$
such that $\cL|_{X \setminus Z}$ is locally free of rank $1$.

Let $f: X \to S$ be a morphism. A coherent sheaf $\cL$ on $X$ is called a \textit{flat family of divisorial sheaves} if $\cL$ is flat over $S$ and its fibers are divisorial sheaves. (Note that $\cL$ need not be a divisorial sheaf on $X$.)
\end{defn}

\begin{rem}\label{rem:univ-hull}
If $\cL$ is a flat family of divisorial sheaves, we apply \cite[Corollary 9.14]{kollarFamiliesVarietiesGeneral2023} to $\cL=F=G$ in \textit{loc.\ cit.}, it follows that $\cL$ is equal to its relative hull $\cL^H$. Moreover,
$\cL$ is a \textit{universal hull} (see \cite[Definition 9.16]{kollarFamiliesVarietiesGeneral2023}). Indeed, for any morphism $q: T \to S$ with induced morphism
$q_X: X \times_S T \to X$, the pull-back $\cL_T:=q_X^* \cL$ is again a flat family of divisorial sheaves,   and hence $\cL_T$ is equal to its relative hull $(\cL_T)^H$.
\end{rem}

\begin{defn}[Mostly flat family of divisorial sheaves]

Using the notation of Definition \ref{def:most flat}, $\cF$ is a \textit{mostly flat family of $S_2$ sheaves} if $\cF|_U$ is flat over $S$ with $S_2$ fibers and $\cF = \cF^H$.

A coherent sheaf $\cL$ on $X$ is a \textit{mostly flat family of divisorial sheaves} if $\cL|_U$ is invertible and $\cL = \cL^H$.

\end{defn}

	\begin{defn}[Relative Mumford divisor]\label{def:relative mumford}
		Let $f:X\to S$ be a flat finite type morphism with $S_2$ fibers of pure dimension $d$. 
		A subscheme $D\subset X$ is a \emph{relative Mumford divisor}   if there is an open set $U\subset X$ such that 		
		\begin{enumerate}
		    
			\item ${\rm codim}_{X_s}(X_s\setminus U_s) \geq 2$ for each $s\in S$, 
			\item $D\vert_U$ is a relative Cartier divisor,
			\item $D$ is the closure of $D\vert_U$, and
			\item $X_s$ is smooth at the generic points of $D_s$ for every $s\in S$.    
		\end{enumerate}
			By $D|_U$ being \textit{relative Cartier} we mean that $D|_U$ is a Cartier divisor on $U$ and that 
		its support does not contain any irreducible component of any fiber $U_s$.  	If $U \subset X$ denotes the largest open set with these properties, then $Z := X \setminus U$ is the non-Cartier locus of $D$.
	\end{defn}

\begin{defn}[Divisorial pullback]
    
If $D \subset X$ is a relative Mumford divisor for $f: X \to S$ and $q:T \to S$ is a morphism, 
then the \emph{divisorial pullback} $D_T$ on $X_T := X \times_S T$ is defined to be the closure 
of the pullback of $D|_U$ to $U_T$. 
In particular, for each $s \in S$, we define $D_s := D|_{X_s}$ to be the closure of $D|_{U_s}$, 
which is the divisorial pullback of $D$ to $X_s$.

Note that if $\cL = \mathcal{O}_X(D)$ is the corresponding divisorial sheaf, then $\cL$ is a mostly flat family of divisorial sheaves on $X$, and $\mathcal{O}_{X_T}(D_T)$ is the hull pull-back $(\cL_T)^H$ on $X_T$.
   \end{defn} 
\begin{defn}[Relative canonical class]
Let $f: X \to S$ be a flat, projective morphism with demi-normal fibers.
The relative canonical sheaf $\omega_{X/S}$ was constructed in \cite[2.68]{kollarFamiliesVarietiesGeneral2023}.
Let $Z \subset X$ be the subset where the fibers are neither smooth nor nodal.
Set $j: U := X \setminus Z \hookrightarrow X$.
Then $X_s \cap Z$ has codimension $\ge 2$ for every fiber $X_s$, and $\omega_{U/S}$ is locally free.
Thus $\omega_{X/S}$ is a mostly flat family of divisorial sheaves with corresponding divisor class $K_{X/S}$.
As in Definition \ref{def:Flat family of divisorial sheaves},
 we define its reflexive powers by
\[
\omega_{X/S}^{[q]}
:= j_*\omega_{U/S}^{\otimes q}
\simeq \mathcal{O}_X(qK_{X/S}).
\]
\end{defn}
If the fibers of $f:X\to S$ are slc, then $\omega_{X/S}$ is a flat family of divisorial sheaves by \cite[Theorem 2.67]{kollarFamiliesVarietiesGeneral2023}. However, its reflexive powers are usually only mostly flat over $S$.

	\begin{defn}[Locally stable family]\label{def:locally stable}
		A \emph{locally stable family of  klt (resp. lc, slc) pairs} $(X,B=\sum_{j=1}^na_jD_j) \to S$ over a reduced Noetherian scheme $S$
		is a flat finite type morphism $X\to S$ with $S_2$ fibers  and a $\bQ$-divisor $B$ on $X$   satisfying
	\begin{enumerate}
\item each prime component $D_j$ of $B$ is a relative Mumford divisor,
\item $K_{X/S}+B$ is $\mathbb{Q}$-Cartier, and
\item $(X_s,B_s)$ is a klt (resp.\ lc, slc) pair for any point $s \in S$.
\end{enumerate}
% A locally stable family of slc pairs is called a \emph{KSBA-stable family} if, moreover,
% \begin{enumerate}\setcounter{enumi}{3}
% \item $K_{X/S}+B$ is ample over $S$.
% \end{enumerate}
	\end{defn}

Let $q:T\to S$ be a morphism between reduced schemes,
and let $f:(X,B)\to S$ be a locally stable family of demi-normal pairs. 
Let $X_T:=X\times_S T\to T$ be the induced morphism, and set $B_T=\sum_{j=1}^n a_j D_{T,j}$, where $D_{T,j}$ is the divisorial pullback of $D_j$. Then $(X_T,B_T)\to T$ is also a locally stable family of pairs by Remark~\ref{rem:univ-hull}.

\subsection{Fiber product  of locally stable families}\label{sec:fiber prod}

Let $(X,B=\sum_{j=1}^n a_jD_j)\to S$ be a locally stable family of klt pairs over a reduced quasi-projective scheme $S$, where $D_j$ are relative Mumford divisors. Let $L$ be a line bundle on $X$.
Let $m$ be a positive integer and denote
$$
X^{(m)}:= \underbrace{X \times_S X \times_S \cdots \times_S X}_{m}
$$
the $m$-fold fiber product of $X$ over $S$, and denote by $p_i$ the projection onto the $i$-th factor. 
Take $U$ to be the open subset of $X$ such that $f|_U$ is smooth. Then $K_{U/S}$ is Cartier, $D_j|_U$ is Cartier, and
\[
\operatorname{codim}_{X_s}(X_s \setminus U_s) \ge 2
\quad \text{for all } s \in S.
\] Define
\[
V := \underbrace{U \times_S U \times_S \cdots \times_S U}_{m},
\]
that is, the $m$-fold fiber product of $U$ over $S$. By definition, $V$ is a Zariski open subset of $X^{(m)}$ with
$\operatorname{codim}_{X^{(m)}_s} (X^{(m)}_s \setminus V_s) \geq 2$ for all $s\in S$.
For each $D_j$, we define $D_j^{(m)}$ to be the closure of 
$$
D_{j,V}^{(m)}:=\sum_{i=1}^m p_i^*(D_j|_U).
$$
We write $B^{(m)}:=\sum_{j=1}^n a_j D_j^{(m)}$ and $L^{(m)}:=\bigotimes_{i=1}^m p_i^*L$.

\begin{prop}\label{prop:fiber product}
Let $f:(X,B=\sum_{j=1}^n a_jD_j)\to S$ be a locally stable family of klt pairs over a reduced quasi-projective scheme $S$, and let $L$ be a line bundle on $X$. For a given positive integer $m$, we have
\begin{enumerate}
    \item $f^{(m)}:(X^{(m)},B^{(m)})\to S$ is also a locally stable family of klt pairs.
    \item For $\alpha\in\bN$ divisible by the Cartier index of $K_{X/S}+B$ and $\beta\in \bN$, we have
    $$f^{(m)}_*\big(\cO_{X^{(m)}}(\alpha(K_{X^{(m)}/S}+B^{(m)}))\otimes (L^{(m)})^\beta\big)\simeq
    \bigotimes^m f_*\big(\cO_X(\alpha(K_{X/S}+B))\otimes L^\beta\big).$$
    \item $\lct(X_s,B_s,|L_s|)=\lct(X^{(m)}_s,B^{(m)}_s,|L^{(m)}_s|)$ for all $s\in S$.
\end{enumerate}
	\end{prop}
\begin{proof}

(1). For each $j$, by construction, $D_j^{(m)}$ is a relative Mumford divisor.
By Lemma~\ref{lem:prod of klt}, 
$$(X_s^{(m)}, B_s^{(m)})= \underbrace{(X_s,B_s) \times (X_s,B_s) \times\cdots \times (X_s,B_s)}_{m}$$ is klt for any $s \in S$.
It remains to show that $K_{X^{(m)}/S} + B^{(m)}$ is $\mathbb{Q}$-Cartier.
We claim that for $q \in \mathbb{N}$ divisible by the Cartier index of $K_{X/S} + B$, we have
\begin{equation}\label{eq:reflexive-pullback}
\mathcal{O}_{X^{(m)}}\bigl(q(K_{X^{(m)}/S} + B^{(m)})\bigr)
\simeq
\bigotimes_{i=1}^m
p_i^*\,
\mathcal{O}_X\bigl(q(K_{X/S} + B)\bigr).
\end{equation}
Assuming this holds, since the right-hand side is Cartier, so is the left-hand side. It follows that $K_{X^{(m)}/S} + B^{(m)}$ is $\mathbb{Q}$-Cartier.

By the induction hypothesis, we have
\begin{equation*}
\mathcal{O}_{X^{(m-1)}}\bigl(q(K_{X^{(m-1)}/S} + B^{(m-1)})\bigr)
\simeq
\bigotimes_{i=1}^{m-1}
p_i^*\,\mathcal{O}_X\bigl(q(K_{X/S} + B)\bigr),
\end{equation*}
where $p_i: X^{(m-1)} \to X$ is the $i$-th projection.
Therefore, it is sufficient to prove that
\begin{equation}\label{eq:main-step}
\begin{aligned}
&\mathcal{O}_{X^{(m)}}\bigl(q(K_{X^{(m)}/S} + B^{(m)})\bigr)\\
\simeq & p_m^*\,\mathcal{O}_X\bigl(q(K_{X/S} + B)\bigr)
\otimes
g^*\,\mathcal{O}_{X^{(m-1)}}\bigl(q(K_{X^{(m-1)}/S} + B^{(m-1)})\bigr).
\end{aligned}
\end{equation}
where $g = (p_1,\dots,p_{m-1}): X^{(m)} \to X^{(m-1)}$.The commutative diagram
\begin{equation*}
\begin{tikzcd}
X^{(m)} \arrow[r,"g"] \arrow[d,"p_m"'] \arrow[rd,"f^{(m)}"] & X^{(m-1)}  \arrow[d,"f^{(m-1)}"]  \\
X \arrow[r,"f"'] & S 
\end{tikzcd}
\end{equation*}
may be helpful.

Note that the right-hand side of equation (\ref{eq:main-step}) is an invertible sheaf, hence a flat family of divisorial sheaves over $S$. The equation (\ref{eq:main-step}) holds over $V$, the $m$-fold fiber product of $U$ over $S$. It follows that the left-hand side of equation (\ref{eq:main-step}) is the relative hull of the right-hand side. By Remark~\ref{rem:univ-hull}, they are equal.

(2). Let $M$ denote the line bunlde $\cO_X(\alpha(K_{X/S}+B))\otimes L^\beta$ on $X$. By equation (\ref{eq:reflexive-pullback}), we know that
$M^{(m)}:=\bigotimes_{i=1}^m p_i^*M$ is equal to $\cO_{X^{(m)}}\bigl(\alpha(K_{X^{(m)}/S}+B^{(m)})\bigr)\otimes (L^{(m)})^\beta$. It is enough to show that
$f^{(m)}_*M^{(m)}\simeq\bigotimes^m f_*M$
via the following computation:
\begin{align*}
f^{(m)}_* M^{(m)}
&\simeq f^{(m-1)}_* g_* M^{(m)} \\
\text{(eq.\ (\ref{eq:main-step}))} \quad
&\simeq f^{(m-1)}_* g_* \bigl(g^* M^{(m-1)} \otimes p_m^* M \bigr) \\
\text{(projection formula)} \quad
&\simeq f^{(m-1)}_* \bigl(M^{(m-1)} \otimes g_* p_m^* M \bigr) \\
\text{(flat base change)} \quad
&\simeq f^{(m-1)}_* \bigl(M^{(m-1)} \otimes (f^{(m-1)})^* f_* M \bigr) \\
\text{(projection formula)} \quad
&\simeq f^{(m-1)}_* M^{(m-1)} \otimes f_* M \\
\text{(induction)} \quad
&\simeq \bigl(\bigotimes^{m-1} f_* M\bigr)\otimes f_* M \\
&\simeq \bigotimes^m f_* M.
\end{align*}

(3). This follows from \cite[Theorem 3.7]{ambroVariationLogCanonical2016} (see also \cite[Proposition 8.11]{kovacsProjectivityModuliSpace2017}).
\end{proof}

\begin{lemma}  \label{lem:prod of klt}   
Let $(X_1,B_1)$ and $(X_2,B_2)$ be two klt pairs, and set $X = X_1 \times X_2$ and $B = (B_1 \times X_2) \cup (X_1 \times B_2)$. Then $(X,B)=(X_1,B_1)\times (X_2,B_2)$ is still a klt pair.
\end{lemma}

\begin{proof}
Let $f_i \colon Y_i \to X_i$ be log resolutions of the two pairs; by assumption on the singularities,
\[
K_{Y_i} \equiv f_i^{*}(K_{X_i} + B_i) + \sum_{j} a_{i,j} E_{i,j}
\]
with $a_{i,j} > -1$. Setting $Y = Y_1 \times Y_2$ and $f = f_1 \times f_2$, the morphism $f \colon Y \to X$ is a log resolution of the pair $(X, B)$. We compute that
\[
K_Y \equiv p_1^{*}K_{Y_1} + p_2^{*}K_{Y_2} \equiv f^{*}(K_{X} + B) + \sum_{j} \bigl(a_{1,j} E_{1,j} \times Y_2 + a_{2,j} Y_1 \times E_{2,j}\bigr),
\]
which shows that $(X, B)$ is indeed klt.
\end{proof}

\subsection{Viehweg's ampleness criterion}\label{sec:Viehweg's ampleness criterion}
In this subsection, we recall Viehweg's idea of using the weak positivity of vector bundles to show that certain schemes are quasi-projective and that certain invertible sheaves are ample.

\begin{defn}
    Let $S$ be a scheme, and let $i:S_0\to S$ be a Zariski open dense subscheme. A locally free sheaf $\cG$ on $S$ is called \textit{weakly positive over $S_0$}, if for all morphisms $\tau:T\to S$ with $T$ a quasi-projective reduced scheme, the sheaf $\tau^*\cG$ is weakly positive over $\tau^{-1}(S_0)$.
\end{defn}
 
Let $\mathcal{E}$ be a sheaf locally free of rank $r$ and weakly positive over $S$. Let $\cQ$ be a locally free quotient of $S^{p}\mathcal{E}$. Consider the inclusion
\[
\epsilon_s:  \operatorname{Ker}(S^{p}\mathcal{E} \to \cQ) \otimes_{\mathcal{O}_S} k(s) \to S^{p}\mathcal{E} \otimes_{\mathcal{O}_S} k(s)\cong S^{p}(k^r)
\]
for a geometric point $s \in S$.
It defines a point $[\epsilon_s]$ in the Grassmann variety $\operatorname{Gr} = \operatorname{Grass}\big(\operatorname{rank}(\cQ), S^{p}(k^r)\big)$, which parametrizes $\operatorname{rank}(Q)$-dimensional quotient spaces of $S^{p}(k^r)$.
The group $G = \mathrm{PGL}(r,k)$ acts on $\operatorname{Gr}$ by changing the basis of $\mathcal{E} \otimes k(s) \cong k^r$. Although $[\epsilon_s]$ depends on the chosen basis for $\mathcal{E} \otimes k(s)$, the $G$-orbit $G_s = G_{[\epsilon_s]}$ of $[\epsilon_s]$ in $\operatorname{Gr}$ is well defined and depends only on $\delta:S^{p}\mathcal{E} \to \cQ$.
\begin{defn}
    
We say that the kernel of $\delta$  has \textit{maximal variation} in $s \in S$ if the set
$\{s' \in S\mid G_{s'} = G_s\}$ is finite and if $\dim(G) = \dim(G_s)$.
\end{defn} 
\begin{thm}[{\cite[Theorem 4.33]{viehwegQuasiprojectiveModuliPolarized1995}}] \label{thm:viehweg ampleness}
    Let $S$ be a scheme, defined over an algebraically closed field $k$
of characteristic zero, and let $\mathcal{E}$ be a locally free and weakly positive sheaf on $S$.
For a surjective morphism $ \delta:S^{p}\mathcal{E} \to \cQ$ to a locally free sheaf $Q$, assume that
the kernel of $ \delta$  has maximal variation in all points $s \in S$.

Then $S$ is a quasi-projective scheme and the sheaf $\cA=\det(\cQ)^a \otimes \det(\mathcal{E})^b$
is ample on $S$ for $b \gg a \gg 0$.

\end{thm}

		%%%%%%%%%%%%%%%%%%%%%%%%%%%%%%%%%%%%%%%%%%%%%%%%	
        
\section{Weak Positivity of Direct Images of Sheaves}

Semipositivity results for direct images of relative log pluri-canonical bundles over normal projective bases have been extensively studied in \cite{fujinoNotesWeakPositivity2017,fujinoSemipositivityTheoremsModuli2018,fujinoVanishingSemipositivityTheorems2020,codogniEffectivePositivityHodge2025}. We recall them here.

\begin{thm}\label{thm:semipositivity}
  Let $f:(X,B)\to S$ be a locally stable family of slc pairs over a normal projective variety $S$ such that $K_X+B$ is $f$-semiample. Let $q\geq 2$ be a natural number such that $qB$ is integral. Then
\begin{enumerate}
    \item for all $i\geq 0$, the sheaf $R^i f_*\cO_X\bigl(q(K_{X/S}+B)\bigr)$ is locally free and commutes with arbitrary base change, and
    \item the sheaf $f_*\cO_X\bigl(q(K_{X/S}+B)\bigr)$ is nef.
\end{enumerate}
\end{thm}
\begin{proof}
    (1). For any morphism $\tau:T \to S$ from a reduced scheme $T$, with induced morphism
$\tau_X : X_T := X \times_S T \to X$, under the assumption that the base $S$ is normal and that $qB$ is integral, it is shown in Equation (3.3) of \cite[Proposition 3.3]{codogniEffectivePositivityHodge2025} that $\cO_X\bigl(q(K_{X/S}+B)\bigr)$ is flat over $S$ and
\[
q_X^*\cO_X\bigl(q(K_{X/S}+B)\bigr)
\simeq \cO_{X_T}\bigl(q(K_{X_T/T}+B_T)\bigr),
\]
where $B_T$ is the divisorial pullback of $B$ via $T \to S$.
It then follows that there exists a natural morphism
\[
\rho:\tau^*R^i f_*\cO_X\bigl(q(K_{X/S}+B)\bigr)\to R^i f_{T,*}\cO_{X_T}\bigl(q(K_{X_T/T}+B_T)\bigr).
\]

We then follow the argument in \cite[Lemma 2.40]{viehwegQuasiprojectiveModuliPolarized1995}. By “Cohomology and Base Change” (see \cite[Theorem \uppercase\expandafter{\romannumeral3}.12.11]{hartshorneAlgebraicGeometry1977}), the second statement follows from the first one. Moreover, assuming that $S$ is affine, one obtains a bounded complex $\mathcal{E}^\bullet$ of locally free coherent sheaves on $S$ such that for every coherent sheaf $\mathcal{G}$ on $S$,
\[
R^i f_*\left(\cO_X\left(q(K_{X/S}+B)\right)\otimes\mathcal{G}\right) \simeq \cH^i\left(\mathcal{E}^\bullet \otimes \mathcal{G}\right).
\]
To show that $\cH^i(\mathcal{E}^\bullet)$ is locally free, it suffices to verify the local freeness of $\cH^i\left(\mathcal{E}^\bullet \otimes \mathcal{G}\right)$, where $\mathcal{G} = \sigma_* \mathcal{O}_C$ for the normalization $\sigma : C \to C'$ of a curve $C' \subset S$.
Indeed, letting $\mathcal{E}_C^\bullet$ denote the pullback of $\mathcal{E}^\bullet$ to $C$, the local freeness of $\cH^i(\mathcal{E}_C^\bullet)$ implies that
\[
h^i(s) = \dim H^i\big(X_s, \cO_{X_s}\left(q(K_{X_s}+B_s)\right)\big)
\]
is constant for $s \in C$.
By varying $C$, one concludes that $h^i(s)$ is constant on $S$, hence $\cH^i(\mathcal{E}^\bullet)$ is locally free.
Using \cite[Lemma 2.39]{viehwegQuasiprojectiveModuliPolarized1995}, we may assume that $S$ is a smooth curve.

Since $(X,B)\to S$ is a locally stable family of slc pairs over a smooth curve, it follows from \cite[Corollary 4.55]{kollarFamiliesVarietiesGeneral2023} that $(X,B)$ is an slc pair. Hence $\Supp(B)$ does not contain any irreducible component of the conductor of $X$.
Moreover, since $q\geq2$, the divisor $q(K_{X/S}+B)-(K_X+B)$ is $f$-semiample. Hence, by \cite{ambroQuasilogVarieties2003}\cite[Theorem 1.12]{fujinoFundamentalTheoremsSemi2014}, every associated prime of $R^i f_*\cO_X\bigl(q(K_{X/S}+B)\bigr)$ is the generic point of the $f$-image of some slc stratum of $(X,B)$ for every $i$.
On the other hand, by \cite[Corollary 4.56]{kollarFamiliesVarietiesGeneral2023}, every slc stratum of $(X,B)$ dominates $S$. Therefore $R^i f_*\cO_X\bigl(q(K_{X/S}+B)\bigr)$ is torsion free, and hence locally free since $S$ is a smooth curve.

    (2). It follows from (1) and \cite[Corollary 3.26.(3)]{codogniEffectivePositivityHodge2025}
\end{proof}

\begin{thm}\label{thm: wp of direct image  of family of gmm'}
Let $f:(X,B)\to S$ be a locally stable family of klt pairs over a reduced quasi-projective scheme $S$ such that $K_X+B$ is $f$-semiample. 
Then for every natural number $q\geq 2$ such that $q(K_{X/S}+B)$ is Cartier, 
\begin{enumerate}
    \item the sheaf $R^if_*\cO_X\bigl(q(K_{X/S}+B)\bigr)$ is locally free  and compatible with arbitrary base change for all $i\geq 0$, and
    \item the sheaf $f_*\cO_X\bigl(q(K_{X/S}+B)\bigr)$ is weakly positive over $S$.
\end{enumerate}
\end{thm}

\begin{proof}
    (1). For any morphism $T \to S$ from a reduced scheme $T$, with induced morphism
$q_X : X_T := X \times_S T \to X$, since $\cO_X\bigl(q(K_{X/S}+B)\bigr)$ is a line bundle, and hence a flat family of divisorial sheaves over $S$, by Remark~\ref{rem:univ-hull} it is a universal hull. Then,
\[
q_X^*\cO_X\bigl(q(K_{X/S}+B)\bigr)
\simeq \cO_{X_T}\bigl(q(K_{X_T/T}+B_T)\bigr),
\]
where $B_T$ is the divisorial pullback of $B$ via $T\to S$.
By the proof of Theorem~\ref{thm:semipositivity}.(1), it is enough to verify the local freeness of $R^i f_*\cO_X\bigl(q(K_{X/S}+B)\bigr)$ when $S$ is a smooth curve.

Since $(X,B)\to S$ is a locally stable family of klt pairs over a smooth curve, it follows from \cite[Corollary 4.56]{kollarFamiliesVarietiesGeneral2023} that $(X,B)$ is klt. Moreover, since $q(K_{X/S}+B)-(K_X+B)$ is $f$-semiample, \cite[Theorem 1.2.7]{kawamataIntroductionMinimalModel1987} implies that $R^i f_*\cO_X\bigl(q(K_{X/S}+B)\bigr)$ is torsion free, and hence locally free since $S$ is a smooth curve.

(2). \textit{Step 1.} In this step, we apply weak semistable reduction and the lc MMP to construct a compactification $\overline{W}$ of a generically finite cover of a desingularization of $S$.
The argument is similar to Step~1 and Step~2 in the proof of \cite[Theorem 3.7]{jiangBoundednessPolarizedLog2025}.

Let $\delta:S'\to S$ be a desingularization of $S$, and let $f':(X',B')\to S'$ be the base change. Then, by \cite[Corollary 4.56]{kollarFamiliesVarietiesGeneral2023}, $(X',B')$ is klt. Let $\pi':Y'\to X'$ be a fiberwise log resolution of $(X',B')$, and let $B_{Y'}$ be the divisor on $Y'$ defined by $K_{Y'}+B_{Y'}=\pi'^*(K_{X'}+B')$.

By \cite[Theorem 2.1, Proposition 4.4, and Proposition 5.1]{abramovichWeakSemistableReduction2000}, there exists a generically finite cover $\tau:W\to S'$ from a smooth variety $W$  and a compactification $W\hookrightarrow \overline{W}$ such that the pullback $\big(Y',\Supp(B_{Y'})\big)\times_{S'} W$ extends to a locally stable morphism $\big(\overline{Y},\Supp(B_{\overline{Y}})\big)\to \overline{W}$, where $\overline{W}$ is smooth and $\overline{W}$ is $\bQ$-factorial.

Write $B_{\overline{Y}}=B_{\overline{Y}}^+-B_{\overline{Y}}^-$. Since the generic fiber of $(\overline{Y},B_{\overline{Y}}^+)\to \overline{W}$ is klt and $\overline{W}$ is smooth, it follows from \cite[Corollary 4.56]{kollarFamiliesVarietiesGeneral2023} that every lc center of $(\overline{Y},B_{\overline{Y}}^+)$ dominates $\overline{W}$, hence $(\overline{Y},B_{\overline{Y}}^+)$ is klt.
Let 
\[f_W:(X_W,B_W)\to W\]
be the pullback of $f:(X,B)\to S$ via $W\to S$. Since $(X_W,B_W)$ is a weak lc model of $(\overline{Y},B_{\overline{Y}}^+)$ over $W$ which is semiample over $W$, the pair $(\overline{Y},B_{\overline{Y}}^+)$ admits a good minimal model over $W$ by \cite[Lemma 2.9.1]{haconBoundednessModuliVarieties2018}. Thus, by \cite{birkarExistenceLogCanonical2012}\cite[Theorem 1.1]{haconExistenceLogCanonical2013}, we can run an MMP for $(\overline{Y},B_{\overline{Y}}^+)$ over $\overline{W}$, which terminates with a good minimal model $(\overline{V},B_{\overline{V}})$ over $\overline{W}$.
By \cite[Corollary 4.57.1]{kollarFamiliesVarietiesGeneral2023}, 
\[
\overline{g}:(\overline{V},B_{\overline{V}})\to \overline{W}
\]
is locally stable. 
\[
\begin{tikzcd}
&(Y', B_{Y'}) 
    \arrow[d,"\pi'"'] 
& (Y',B_{Y'}) \times _{S'}W
    \arrow[l]
    \arrow[r,hook]
    \arrow[d,dashed]
& (\overline{Y}, B_{\overline{Y}}) 
\arrow[rd, dashed]
    \arrow[dd, bend right=60] \\
(X,B) \arrow[d, "f"'] 
& (X', B') \arrow[d, "f'"'] 
    \arrow[l] 
& (X_W, B_W) 
    \arrow[l] 
    \arrow[d, "f_W"'] 
  &  (V, B_V) 
    \arrow[ld, "g"'] 
    \arrow[r, hook] 
& (\overline{V}, B_{\overline{V}}) 
    \arrow[ld, "\overline{g}"] 
\\
S   
& S' 
    \arrow[l, "\delta"'] 
& W 
    \arrow[l,"\tau"'] 
    \arrow[r, hook] 
& \overline{W}
\end{tikzcd}
\]

\textit{Step 2.} In this step, we show that there exists a nef locally free sheaf $\cG$ on $\overline{W}$ such that $\cG|_W \simeq \tau^*\delta^*\cF$, where $\cF := f_*\cO_X\bigl(q(K_{X/S}+B)\bigr)$.

By the negativity lemma \cite[Lemma 3.39]{kollarBirationalGeometryAlgebraic1998}, the restriction $g:(V,B_V)\to W$ of $\overline{g}:(\overline{V},B_{\overline{V}})\to \overline{W}$ to $W$ is crepant birational to $f_W:(X_W,B_W)\to W$.
That is, let $Y$ be a common resolution of $X_W$ and $V$, and let $p:Y\to X_W$ and $r:Y\to V$ be the induced morphisms. Then we have
\[
p^*(K_{X_W/W}+B_W)=r^*(K_{V/W}+B_V).
\]

Since $f_W:(X_W,B_W)\to W$ is the pullback family of $f$ via $W\to S$, it follows from Remark~\ref{rem:univ-hull} that $q(K_{X_W/W}+B_W)$ is Cartier. Hence
\begin{equation}\label{eq:crepant-pullback}
r^*\bigl(q(K_{V/W}+B_V)\bigr)
=
p^*\bigl(q(K_{X_W/W}+B_W)\bigr)
\end{equation}
is Cartier. Therefore $q(K_{V/W}+B_V)$ is Cartier by \cite[Lemma 3.1]{hashizumeNoteLctrivialFibrations2024}, since $(V,B_V)$ is klt.
Consequently, the divisor $q(K_{V/W}+B_V)$ extends to a Weil divisor
$q(K_{\overline{V}/\overline{W}}+B_{\overline{V}})$
on $\overline{V}$.
Let 
$$\cG := \overline{g}_*\cO_{\overline{V}}\bigl(q(K_{\overline{V}/\overline{W}}+B_{\overline{V}})\bigr).$$ Then, by Theorem~\ref{thm:semipositivity}, $\cG$ is a nef locally free sheaf on $\overline{W}$, and it commutes with arbitrary base change. 
Therefore we have
\[
\tau^*\delta^*\cF
\simeq
f_{W,*}\cO_{X_W}\bigl(q(K_{{X_W}/W}+B_W)\bigr)
\simeq
g_*\cO_V\bigl(q(K_{V/W}+B_V)\bigr)
\simeq
\cG|_W,
\]
where the first and third isomorphisms are the base change isomorphisms in (1) and Theorem~\ref{thm:semipositivity}.(1), respectively. The second isomorphism follows from the projection formula, since both sides of equation~\eqref{eq:crepant-pullback} are Cartier.

\textit{Step 3.} In this step, we construct a finite cover $\phi:T\to S$ inductively using a stratification of $S$.
\begin{claim}\label{claim}
There exists a chain of reduced closed subschemes of $S$
\[
S=S^1\supset S^2\supset\cdots\supset S^m\supset S^{m+1}=\emptyset
\]
and a morphism of reduced schemes $\phi:T\to S$ such that
\begin{enumerate}
    
    \item[(a)] $\phi:T\to S$ is a finite cover which admits a splitting trace map.
\item[(b)] $S^i-S^{i+1}$ is nonempty and nonsingular.

\item[(c)] For each $i$, there exists a desingularization $\delta_i:S'^i\to S^i$, a generically finite cover $\tau_i:W^i\to S'^i$, and a compactification $W^i\subset\overline{W^i}$ such that there exists a nef locally free sheaf $\mathcal{G}_{\overline{W^i}}$ on $\overline{W^i}$, where $\mathcal{G}_{\overline{W^i}}|_{W^i}$ is isomorphic to the pullback of $\mathcal{F}|_{S^i}$ to $W^i$. Moreover, $\mathcal{G}_{\overline{W^i}}$ is compatible with further pullback.

    \item[(d)] The restriction of $\phi$ to $\phi^{-1}(S^i-S^{i+1})$ factors through
\[
\begin{tikzcd}
\phi^{-1}\left(S^i - S^{i+1}\right)
\arrow[r]
\arrow[d, hook]
&
\tau_i^{-1}\delta_i^{-1}\left(S^i - S^{i+1}\right)
\arrow[r]
\arrow[d, hook]
&
S^i - S^{i+1}
\arrow[d, hook]
\\
T
&
W^i
\arrow[r, "\delta_i \circ \tau_i"]
&
S^i.
\end{tikzcd}
\]
\end{enumerate}
\end{claim}
\begin{proof}[Proof of the Claim]
Let $S\subset\overline{S}$ be any compactification, and choose an embedding $\overline{S}\subset\mathbb{P}=\mathbb{P}^M$. 
We will construct $T\to S$ by constructing a finite morphism $\Psi:\mathbb{P}'\to \mathbb{P}$ with $\mathbb{P}'$ normal such that $T=\Psi^{-1}(S)$ and $\phi=\Psi|_T$. Then $T\to S$ is automatically finite and admits a splitting trace map by Lemma~\ref{lem:trace}. Hence condition {\rm(a)} holds.

We start with $S^1=S$.
Assume that for some $j$, the schemes $S^1,\ldots,S^j$ have been constructed.
By induction, we have constructed $S'^i$, $W^i$, and $\overline{W^i}$ for $i<j$, together with a finite morphism of schemes
\[
\Psi^{j-1}:\mathbb{P}'^{\,j-1}\to\mathbb{P}
\]
such that conditions {\rm(b)}, {\rm(c)}, and {\rm(d)} hold for $1\leq i\leq j-1$, with $\phi^{j-1}:=\Psi^{j-1}|_{T^j}$ in place of $\phi$, where $T^j=(\Psi^{j-1})^{-1}(S)$.

We may take a desingularization $\delta_j:S'^j\to S^j$, a generically finite cover $\tau_j:W^j\to S'^j$, and a compactification $W^j\subset\overline{W^j}$ as in Step~1 and Step~2 such that condition  {\rm(c)} holds for $i=j$. Consider the morphism $\lambda_j:W^j\to\mathbb{P}$, by Lemma~\ref{lem: dominant cover}, there exists an irreducible normal covering $\widetilde{\Psi}:\widetilde{\mathbb{P}}'\to \mathbb{P}$ and a scheme $\widetilde{W}^j$ birational to $W^j$ such that $\widetilde{\Psi}^{-1}(S^j)\to S^j$ factors through a morphism $\widetilde{\tau}:\widetilde{W}^j\to S^j$.

For $\Psi^{j}:\mathbb{P}'^{\,j}\to \mathbb{P}$, we take the morphism obtained by normalizing
\[
\mathbb{P}'^{\,j-1}\times_{\mathbb{P}}\widetilde{\mathbb{P}}'.
\]
Since $\Psi^{j}$ factors through $\Psi^{j-1}$, condition {\rm(d)} remain valid for $i<j$ when replacing $\Psi^{j-1}$ by $\Psi^{j}$.
To ensure condition {\rm(d)} for $i=j$, we choose $S^{j+1}$ to be the smallest reduced closed subscheme of $S^{j}$ containing the singular locus of $S^{j}$ such that the birational map between $\widetilde{W}^{j}$ and $W^{j}$ induces an isomorphism
\[
\widetilde{\tau}^{-1}(S^{j}-S^{j+1})
\xrightarrow{\cong}
(\tau_j\circ \delta_j)^{-1}(S^{j}-S^{j+1}).
\]
Thus, conditions {\rm(b)}, {\rm(c)} and {\rm(d)} hold for all $i\le j$.

Since $S^{j}\neq S^{j+1}$, the $S^{i}$ form a strictly decreasing chain of closed subschemes. After finitely many steps, we reach $S^{m+1}=\emptyset$, and the construction terminates.
\end{proof}

\textit{Step 4.} In this step, we apply Gabber's Extension Theorem to conclude the proof. 

Let $\phi:T\to S$ be the finite cover obtained in Step~3, and let $W\to S$ be the generically finite cover with compactification $W\subset\overline{W}$ constructed in Step~1. We choose a desingularization $T'\to T$ and a smooth projective compactification $T'\subset\overline{T'}$ dominating $W$ and $\overline{W}$ respectively. We then obtain the following commutative diagram.
\[
\begin{tikzcd}[row sep=1.5em, column sep=3.5em]
T \arrow[d, "\phi"']
&
T' \arrow[l, "\xi"] \arrow[d, "\theta"']
\arrow[r, hook]
&
\overline{T'} \arrow[d, "\overline{\theta}"']
\\
S
&
W \arrow[l, "\tau\circ\delta"'] \arrow[r, hook]
&
\overline{W}
\end{tikzcd}
\]

In Step~2, we obtain a nef locally free sheaf
$\mathcal{G}= \overline{g}_*\mathcal{O}_{\overline{V}}\bigl(q(K_{\overline{V}/\overline{W}}+B_{\overline{V}})\bigr)$
on $\overline{W}$ such that $\mathcal{G}|_W$ is isomorphic to
$\delta^*\tau^*\mathcal{F}$, where $\mathcal{F}:= f_*\mathcal{O}_X\bigl(q(K_{X/S}+B)\bigr)$.
By base change in Theorem~\ref{thm:semipositivity}.(1),
\[\mathcal{G}_{\overline{T'}}
:=
\overline{g}_{\overline{T'}*}
\mathcal{O}_{\overline{V}_{\overline{T'}}}
\bigl(q(K_{\overline{V}_{\overline{T'}}/\overline{T'}}+B_{\overline{V}_{\overline{T'}}})\bigr)\simeq \overline{\theta}^*\mathcal{G}\] 
is a  nef locally free sheaf on $\overline{T'}$,
where  $\overline{g}_{\overline{T'}}:(\overline{V}_{\overline{T'}},B_{\overline{V}_{\overline{T'}}}) \to \overline{T'}$ is the pullback family of $\ov{g}:(\overline{V},B_{\overline{V}})\to \ov{W}$ via $\ov{T'}\to \ov{W}$.
By the base change properties of the corresponding sheaves in (1) and Theorem~\ref{thm:semipositivity}.(1), we have
\[
\mathcal{F}_{T'}=\xi^*\phi^*\mathcal{F}\simeq \theta^*\delta^*\tau^*\mathcal{F}\simeq \theta^*(\mathcal{G}|_W)\simeq (\overline{\theta}^*\mathcal{G})|_{T'}\simeq (\mathcal{G}_{\overline{T'}})|_{T'}.
\]
Therefore, requirements (1) and (2) of Set-up~\ref{setup} are verified. In Step~3, let $\overline{T}:=\Psi^{-1}(\overline{S})$, and after further replacing $\overline{T'}$ by a pullback, we may assume that there exists a morphism $\overline{\xi}:\overline{T'}\to \overline{T}$. We have the following diagram as in Theorem~\ref{thm:Gabber extension}.
\[
\begin{tikzcd}[row sep=1.5em, column sep=3.5em]
T' \arrow[r, "\subset"] \arrow[d, "\xi"']
& \overline{T'} \arrow[d, "\overline{\xi}"'] \\
T \arrow[r, "\subset"] \arrow[d, "\phi"']
& \overline{T} \arrow[r, "\subset"] \arrow[d, "\overline{\phi}"']
& \mathbb{P}' \arrow[d, "\Psi"] \\
S \arrow[r, "\subset"]
& \overline{S} \arrow[r, "\subset"]
& \mathbb{P}
\end{tikzcd}
\]
In order to apply Gabber's Extension Theorem to obtain the weak positivity of $\mathcal{F}$, it remains to verify requirement (3) of Set-up~\ref{setup}. We use the chain of closed subschemes
\[
S=S^1\supset S^2\supset\cdots\supset S^m\supset S^{m+1}=\emptyset
\]
constructed in Step~3.

Fix a morphism $\overline{\eta}: \overline{C} \to \mathbb{P}'$ from a nonsingular projective curve $\overline{C}$, with $C = \overline{\eta}^{-1}(T)$ dense in $\overline{C}$.
For some $j>0$, the image of $\phi\circ\eta(C)$ is contained in $S^j$ but not in $S^{j+1}$. By Claim~\ref{claim}.(d), the morphism $C\to S^j$ factors through $C\to W^j$, and hence extends to $\overline{C}\to \overline{W^j}$.
By Claim~\ref{claim}.(c),  we obtain a nef locally free sheaf $\mathcal{G}_{\overline{W^j}}$ such that $\mathcal{G}_{\overline{W^j}}|_{W^j}$ is isomorphic to the pullback of $\mathcal{F}|_{S^j}$.
Then we define $\mathcal{G}_{\overline{C}}$ to be the  pullback of   $\mathcal{G}_{\overline{W^j}}$ via $\overline{C}\to \overline{W^j}$. By the base change properties in (1) and Theorem~\ref{thm:semipositivity}.(1) again, $\mathcal{G}_{\overline{C}}$ is the locally free extension of $\eta^*\phi^*\mathcal{F}$.
\[\begin{tikzcd}
\overline{C}  & C \arrow[l, "\supset"'] \arrow[rr,bend left=20]\arrow[r,"\eta"] \arrow[rd] & \phi^{-1}(S^{j}-S^{j+1})\subset T \arrow[d,"\phi"] &(\tau_j\circ \delta_j)^{-1}(S^{j}-S^{j+1})\subset W^j\arrow[ld]\\
& & S^j-S^{j+1}
\end{tikzcd}\]
  If $\overline{\eta}': \overline{C'} \to \mathbb{P}'$ factors through $\overline{\gamma}: \overline{C'}\to\overline{C}$, we apply Claim~\ref{claim}.(c) to $\overline{C'}\to\overline{C}\to \overline{W^j}$, and obtain Set-up~\ref{setup}.(3.a). 
If $\overline{\eta}: \overline{C} \to \mathbb{P}'$ lifts to a morphism $\overline{\chi}: \overline{C} \to \overline{T'}$, we apply Claim~\ref{claim}.(c) to $\overline{C}\to\overline{T'}\to\overline{W^1}=\overline{W}$, and Set-up~\ref{setup}.(3.b) holds. We finish the proof.
\end{proof}

% \begin{rem}
%  At the time \cite{viehwegWeakPositivityStability1990,viehwegWeakPositivityStability1990a,viehwegQuasiprojectiveModuliPolarized1995} were written, the Weak Semistable Reduction Theorem \cite{abramovichWeakSemistableReduction2000} was not yet known. Therefore, Viehweg could only use Gabber’s Extension Theorem together with variations of Hodge structures to prove weak positivity for pushforwards of relative canonical bundles twisted by semiample line bundles, and then reduce the weak positivity of pushforwards of powers of relative canonical bundles to this case. In \cite{viehwegCompactificationsSmoothFamilies2010}, using \cite{abramovichWeakSemistableReduction2000}, he proved weak positivity simultaneously for all powers of relative canonical bundles. 

% However, at that time, the MMP for lc pairs \cite{birkarExistenceLogCanonical2012,haconExistenceLogCanonical2013} was not yet available. As a result, verifying the requirements of Gabber's Extension Theorem and analyzing the base change properties of the relevant sheaves required lengthy arguments; see \cite[Comments 5.9]{viehwegCompactificationsSmoothFamilies2010} for a summary of the strategy. Moreover, multiplier ideal sheaves also played an essential role in the argument. Working in the setting of log pairs, some parts of Viehweg's argument can be simplified.
% \end{rem}

\begin{thm}\label{thm:Viehweg package'}
Let $f:(X,B)\to S$ be a locally stable family of klt pairs over a reduced quasi-projective scheme $S$ such that $K_X+B$ is $f$-semiample. 
Let $L$ be an $f$-semiample line bundle on $X$.
Assume that:
\begin{itemize}
 % \item    $c(K_{X/S}+B)$ is Cartier for some natural number $c\geq 2$,
    \item there exists $l\in \bN$ such that 
$\lct(X_s,B_s,|L_s|)>\frac{1}{l}$ for all $s\in S$, and
\item the sheaf $f_*L$ is locally free of rank $r>0$ and compatible with arbitrary base change.
\end{itemize}

Then  for every natural number $q\geq 2$ such that $q(K_{X/S}+B)$ is Cartier,   we have:
\begin{enumerate}
    \item 
    % {\normalfont (\textbf{Base change and locally freeness})} 
    the sheaf
    $R^if_*\bigl(\cO_X(ql(K_{X/S}+B))\otimes L^q\bigr)$
    is locally free 
    %of rank $r'$ 
    and compatible with arbitrary base change  for all $i\geq0$, and

    \item
    % {\normalfont (\textbf{Weak positivity})} 
     the sheaf
    $$
    \left(\bigotimes^{r} f_*\bigl(\cO_X(ql(K_{X/S}+B))\otimes L^q\bigr)\right)\otimes\det(f_*L)^{-q}
    $$
    is weakly positive over $S$.

%   \item {\normalfont (\textbf{Weak stability})}  For $p,q \in \mathbb{N}$ divisible by $c$,  if $r' \neq 0$, then there exists a positive rational number $\delta$ such that
% \[
% \begin{aligned}
% &\left(\bigotimes^{r} f_*\bigl(\cO_X(pl(K_{X/S}+B))\otimes L^p\bigr)\right)\otimes\det(f_*L)^{-p} \\
% &\succeq\ 
% \delta \cdot \det\!\left(f_*\bigl(\cO_X(ql(K_{X/S}+B))\otimes L^q\bigr)\right)^{r}
% \otimes \det(f_*L)^{-q r'}.
% \end{aligned}
% \]
\end{enumerate}
\end{thm}
\begin{proof}
(1). For any morphism $T \to S$ from a reduced scheme $T$, with induced morphism
$q_X : X_T := X \times_S T \to X$, since $\cO_X\bigl(ql(K_{X/S}+B)\otimes L^q\bigr)$ is a flat family of divisorial sheaves over $S$, by Remark~\ref{rem:univ-hull} it is a universal hull. Then,
\[
q_X^*\cO_X\bigl(ql(K_{X/S}+B)\otimes L^q\bigr)
\simeq \cO_{X_T}\bigl(ql(K_{X_T/T}+B_T)\otimes L_T^q\bigr),
\]
where $B_T$ is the divisorial pullback of $B$ via $T\to S$ and $L_T=q_X^*L$. By the proof of Theorem~\ref{thm:semipositivity}.(1), it is enough to verify the torsion freeness of $R^i f_*\bigl(\cO_X(ql(K_{X/S}+B))\otimes L^q\bigr)$ when $S$ is a smooth curve.
Since $(X,B)\to S$ is a locally stable family of klt pairs over a smooth curve, it follows from \cite[Corollary 4.56]{kollarFamiliesVarietiesGeneral2023} that $(X,B)$ is klt. Moreover, since $ql(K_{X/S}+B)+qL-(K_X+B)$ is $f$-semiample, \cite[Theorem 1.2.7]{kawamataIntroductionMinimalModel1987} implies the desired torsion freeness.

 (2). By Lemma \ref{lem:wp functorial}.(2), we are allowed to replace $S$ by a finite cover $\tau : S' \to S$, provided that the trace map splits the inclusion $\mathcal{O}_S \to \tau_*\mathcal{O}_{S'}$. By \cite[Lemma 2.1]{viehwegQuasiprojectiveModuliPolarized1995}, we may assume that for some invertible sheaf $\lambda$ on $S$ one has $\lambda^{r}=\det(f_*L)$. Replacing $L$ by $L \otimes f^*\lambda^{-1}$ does not affect the assumptions or conclusions. Hence we may assume that $\det(f_*L) = \mathcal{O}_S$.
Under this additional assumption, by \cite[Corollary 2.20]{viehwegQuasiprojectiveModuliPolarized1995}, it is enough to show that:
for $q \geq 2$ divisible by the Cartier index of $K_{X/S}+B$, the sheaf
    \[
    f_*\bigl(\cO_X(ql(K_{X/S}+B))\otimes L^q\bigr)
    \]
    is weakly positive over $S$.
% \begin{enumerate}
% \setcounter{enumi}{1}
% \renewcommand{\labelenumi}{(\arabic{enumi}$'$)}
%     \item For $q \in \mathbb{N}$ divisible by $c$, the sheaf
%     \[
%     f_*\bigl(\cO_X(ql(K_{X/S}+B))\otimes L^q\bigr)
%     \]
%     is weakly positive over $S$.
    
%     \item For $p,q \in \mathbb{N}$ divisible by $c$, if $r' \neq 0$, there exists a positive rational number $\delta$ such that
%     \[
%     f_*\bigl(\cO_X(pl(K_{X/S}+B))\otimes L^p\bigr) \succeq \delta \cdot \det\big(f_*\bigl(\cO_X(ql(K_{X/S}+B))\otimes L^q\bigr)\big).
%     \]
% \end{enumerate}

% (2$'$). 
With the notation as in \S\ref{sec:fiber prod}, let $f^{(m)}:(X^{(m)},B^{(m)})\to S$ be the $m$-fold fiber product of $(X,B)$ over $S$, and let $L^{(m)}:=\bigotimes_{i=1}^m p_i^*L$.
Let $m=r=\operatorname{rank}(f_*L)$. Then the determinant gives an inclusion
\[
\det(f_*L)=\mathcal{O}_S
 \hookrightarrow
f^{(r)}_*L^{(r)} \simeq \bigotimes^r f_*L,
\]
which splits locally. Consequently, the zero divisor $\Gamma^{(r)}$ of the induced section of $L^{(r)}$ does not contain any fiber of $f^{(r)}$.
We therefore obtain
\begin{align*}
&f^{(r)}_*\Bigl(\mathcal{O}_{X^{(r)}}\bigl(ql(K_{X^{(r)}/S}+B^{(r)})\bigr)\otimes (L^{(r)})^q\Bigr)\\
\simeq{}& f^{(r)}_*\mathcal{O}_{X^{(r)}}\bigl(ql(K_{X^{(r)}/S}+B^{(r)})+q\Gamma^{(r)}\bigr) \\
\simeq{}& f^{(r)}_*\mathcal{O}_{X^{(r)}}\bigl(ql(K_{X^{(r)}/S}+B^{(r)}+\tfrac{1}{l}\Gamma^{(r)})\bigr).
\end{align*}
By Proposition~\ref{prop:fiber product}.(3), $(X_s^{(r)},B_s^{(r)}+\frac{1}{l}\Gamma_s^{(r)})$ is klt for any $s\in S$, and hence $$(X^{(r)},B^{(r)}+\frac{1}{l}\Gamma^{(r)})\to S$$ is a  family of klt pairs such that $(X^{(r)},B^{(r)}+\frac{1}{l}\Gamma^{(r)})$ is $f^{(r)}$-semiample. It follows that, by 
% Theorem~\ref{thm: wp of direct image of stable family},
 Theorem \ref{thm: wp of direct image  of family of gmm'}.(2), $$f^{(r)}_*\cO_{X^{(r)}}\bigl(ql(K_{X^{(r)}/S}+B^{(r)}+\frac{1}{l}\Gamma^{(r)})\bigr)$$
is weakly positive over $S$.
On the other hand, by Proposition~\ref{prop:fiber product}.(2), we have
\[
\bigotimes_{i=1}^r f_*\bigl(\cO_X(ql(K_{X/S}+B))\otimes L^q\bigr)
\simeq f^{(r)}_*\bigl(\cO_{X^{(r)}}(ql(K_{X^{(r)}/S}+B^{(r)}))\otimes (L^{(r)})^q\bigr),
\]
which is therefore weakly positive over $S$.
Finally, since $f_*\bigl(\cO_X(ql(K_{X/S}+B))\otimes L^q\bigr)$ is locally free by (1), it follows from \cite[Lemma~2.16.d]{viehwegQuasiprojectiveModuliPolarized1995} that $f_*\bigl(\cO_X(ql(K_{X/S}+B))\otimes L^q\bigr)$ is weakly positive over $S$.
% (3$'$). Let $m=r\cdot r'=\operatorname{rank}(f_*L)\cdot\operatorname{rank}f_*\bigl(\cO_X(ql(K_{X/S}+B))\otimes L^q\bigr)$. One has natural inclusions, splits locally,
% \[
% \mathcal{O}_S=\det(f_*L)^{r'}
% \hookrightarrow 
% f^{(m)}_*L^{(m)} \simeq \bigotimes^m f_*L
% \]
% and
% \[
% \begin{aligned}
% \mathcal{A}
% &:= \det\Bigl(f_*\bigl(\mathcal{O}_X(ql(K_{X/S}+B))\otimes L^q\bigr)\Bigr)^r \\
% &\hookrightarrow f^{(m)}_*\Bigl(\mathcal{O}_{X^{(m)}}\bigl(ql(K_{X^{(m)}/S}+B^{(m)})\bigr)\otimes (L^{(m)})^q\Bigr) \\
% &\simeq \bigotimes^m f_*\bigl(\mathcal{O}_X(ql(K_{X/S}+B))\otimes L^q\bigr).
% \end{aligned}
% \]
% Let $\Gamma_1^{(m)}$ and $\Gamma_2^{(m)}$ be the zero divisors of the induced sections. Then we have 
% \[\Gamma_1^{(m)}\sim L^{(m)}\] 
% and 
% \[\Gamma_2^{(m)}+(f^{(m)})^*\cA\sim ql(K_{X^{(m)}/S}+B^{(m)})+q\Gamma_1^{(m)}.\]
% Let $\alpha$ be a positive integer. 
\end{proof}
		%%%%%%%%%%%%%%%%%%%%%%%%%%%%%%%%%%%%%%%%%%%%%%%

\section{Ample Line Bundles on Algebraic Moduli Spaces}
In this section, we use Viehweg's ampleness criterion and the weak positivity results from the previous section to prove that the normalization of the moduli space of polarized good minimal models is quasi-projective.

%%%%%%%%%%%%%%%%%%%%%%%%%%%%%%%%%%%%%%%%%

\begin{defn}[Polarized good minimal models]\label{def:pgmm} 
A \emph{polarized good minimal model} $(X,A)$ over an algebraically closed field $k$ of characteristic zero consists of a projective connected variety $X$ and a line bundle $A$ such that
\begin{enumerate}
    \item $X$ is klt,
    \item $K_X$ is semiample and defines a contraction $f:X\to Z$, and
    \item $A$ is ample over $Z$.
\end{enumerate}

Let $d\in \bN$, $u\in \bQ^{>0}$, and $\sigma\in \bQ[t]$ be a polynomial. A \emph{$(d,u,\sigma)$-polarized good minimal model} is a polarized good minimal model $(X,A)$ such that
\begin{enumerate}
    \item $\dim X=d$,
    \item $\vol(A|_F)=u$, where $F$ is a general fiber of $f:X\to Z$, and
    \item $(K_X+tA)^d=\sigma(t)$.
\end{enumerate}
Let $\mathcal{G}_{klt}(d,u,\sigma)$ denote the set of $(d,u,\sigma)$-polarized good minimal models.
\end{defn}

  \begin{thm}[{\cite{jiangBoundednessKltGood2023}}]\label{thm:boundedness}
The varieties $X$ appearing in $\mathcal{G}_{klt}(d,u,\sigma)$ form a bounded family. Moreover, there exist positive natural numbers $m$ and $\tau$, depending only on $(d,u,\sigma)$, such that $m(\tau K_X + A)$ is very ample.
\end{thm}
	%%%%%%%%%%%%%%%%%%%%%%%%%%%%%%%%%%%%%%%%%%% 

	We  define families of polarized good minimal models and the corresponding moduli functor.

	\begin{defn}[Moduli functor of polarized good minimal models]\label{def:moduli functor}
Let $S$ be a reduced scheme over $k$.
\begin{enumerate}
    \item When $S=\Spec K$ for a field $K$, we define a polarized good minimal model over $K$ as in Definition~\ref{def:pgmm} by replacing $k$ with $K$ and replacing connected with geometrically connected. Similarly we define $(d,u,\sigma)$-polarized good minimal models over $K$.

    \item For general $S$, a \emph{family of polarized good minimal models} over $S$ consists of a projective morphism $X\to S$ and a line bundle $A$ on $X$ such that
    \begin{itemize}
        \item $X\to S$ is a locally stable family,
        \item $(X_s,A_s)$ is a polarized good minimal model over $k(s)$ for every $s\in S$.
    \end{itemize}
    Here $X_s$ is the fiber of $X\to S$ over $s$. Moreover, $K_{X_s}$ is semiample, defining a contraction $X_s\to Z_s$, and $A_s$ is ample over $Z_s$. We denote this family by $(X,A)\to S$.

    \item Let $d\in \mathbb{N}$,  $u\in \mathbb{Q}^{>0}$, and $\sigma\in \mathbb{Q}[t]$ be a polynomial. A \emph{family of $(d,u,\sigma)$-polarized good minimal models} over $S$ is a family $(X,A)\to S$ such that
    $(X_s,A_s)$ is a $(d,u,\sigma)$-polarized good minimal model over $k(s)$ for every $s\in S$.
 
\item We define the moduli functor $\mathfrak{G}_{klt}(d,u,\sigma)$ of $(d,u,\sigma)$-polarized good minimal models
from the category of reduced $\mathbbm{k}$-schemes to the category of groupoids
by choosing:
\begin{itemize}
    \item On objects: for a reduced $\mathbbm{k}$-scheme $S$, one take
    \begin{equation}\nonumber
        \begin{aligned}
        & \mathfrak{G}_{klt}(d,u,\sigma)(S)\\=
        & \{\text{family of } (d,u, \sigma)\text{-polarized good minimal models over } S\}.
        \end{aligned}
    \end{equation}

    We define an isomorphism $(f':X',A'\to S)\to (f:X,A\to S)$ of any two objects in $\mathfrak{G}_{klt}(d,u,\sigma)(S)$ to be an isomorphism
    $\alpha_X:X'\to X$ over $S$ such that $A'\sim_S \alpha_X^*A$.

    \item On morphisms: $\big(f_T:(X_T,A_T)\to T\big)\to \big(f:(X,A)\to S\big)$ consists of morphisms of reduced $k$-schemes $\alpha:T\to S$ such that the natural map $g:X_T\to X\times_S T$ is an isomorphism, and $A_T\sim_T g^*\alpha_X^*A$.
    Here $\alpha_X:X\times_S T\to X$ is the base change of $\alpha$.
\end{itemize}
\end{enumerate}
\end{defn}

	\begin{thm}[{\cite[Theorem A.5]{jiangBoundednessKltGood2023}}]\label{thm:moduli}
		$	\mathfrak{G}_{klt}(d,u,\sigma)$  is a separated Deligne-Mumford stack of finite type, which admits a coarse moduli space $G_{klt}(d,u,\sigma)$ as  a separated algebraic space.
	\end{thm}

\begin{thm}\label{thm:quasi moduli'}
 Let $\delta:\widetilde{G}_{klt}(d,u,\sigma)\to G_{klt}(d,u,\sigma)$ be the normalization of $G_{klt}(d,u,\sigma)$, then $\widetilde{G}_{klt}(d,u,\sigma)$ is a quasi-projective scheme. 
Moreover, if the non-normal locus of $G_{klt}(d,u,\sigma)$  is proper, then  $G_{klt}(d,u,\sigma)$ is a quasi-projective scheme.
\end{thm}

\begin{proof}

By \cite[Theorem~9.25]{viehwegQuasiprojectiveModuliPolarized1995}, there exists a finite cover $\tau:S \to G_{klt}(d,u,\sigma)$ from a reduced normal scheme $S$ such that $\widetilde{G}_{klt}(d,u,\sigma)$ is the quotient of $S$ by a finite group, and there exists a universal family $(f:(X,A)\to S)\in\mathfrak{G}_{klt}(d,u,\sigma)$.
By boundedness \cite{jiangBoundednessKltGood2023}, there exist $n_1,n_2\in \bN$ such that $L:=\cO_X(n_1K_{X/S})\otimes A^{n_2}$ is an $f$-very ample line bundle on $X$ and has no higher cohomology. By \cite{ambroVariationLogCanonical2016} (see also \cite[Lemma~8.10]{kovacsProjectivityModuliSpace2017}), there exists $l\in \bN$ such that $\lct(X_s,B_s,|L_s|)>\frac{1}{l}$ for all $s\in S$. By \cite{jiangBoundednessKltGood2023} again, we can further choose $q\in \bN$ such that $\cO_X(qlK_{X/S})\otimes L^q$ is also an $f$-very ample line bundle on $X$ and has no higher cohomology. Then the sheaves $f_*L$ and $f_*\big(\cO_X(qlK_{X/S})\otimes L^q\big)$ are locally free and compatible with arbitrary base change by \cite[Theorem \uppercase\expandafter{\romannumeral3}.12.11]{hartshorneAlgebraicGeometry1977}. Assume moreover that $q$ is divisible by the rank $r$ of $f_*L$, then by Theorem~\ref{thm:Viehweg package}.(2) and \cite[Lemma~2.16.d]{viehwegQuasiprojectiveModuliPolarized1995}, we obtain that
$$
\cE:=f_* \big(\cO_X(qlK_{X/S})\otimes L^q \big)\otimes\det(f_*L)^{-\frac{q}{r}}
$$
is weakly positive over $S$.

Let $\cK^{(p)}$ be the kernel of the multiplication map from $S^p\cE$ to
$$\cQ:=f_* \big(\cO_X(pqlK_{X/S})\otimes L^{pq} \big)\otimes\det(f_*L)^{-\frac{pq}{r}}.$$
Then, for $s\in S$, the elements of $\cK^{(p)}\otimes k(s)$ are exactly the degree $p$ equations of the embedding $X_s\to \bP^{r'}$, where $r'$ is the rank of $\cE$. For $p\gg 0$, $\cK^{(p)}\otimes k(s)$ recovers the fiber $X_s$ and $L^q|_{X_s}$. 
Since the $q$-torsion elements in $\Pic(X_s)$ are finite, it also recovers $L|_{X_s}$. 
Let $G=\operatorname{PGL}(r',k)$ be the structure group of $\cE$, then $G$ acts on the Grassmannian
$\operatorname{Gr}=\operatorname{Grass}\big(\operatorname{rank}(\cQ), S^{p}(k^{r'})\big)$.
As in $\S \ref{sec:Viehweg's ampleness criterion}$, 
if $G_s$ denotes the orbit of $s$, then $\{z' \in Z \mid G_z = G_{z'}\}$ is isomorphic to $\tau^{-1}(\tau(z))$
and is therefore finite. Since the automorphism group of $(X_s,L_s)$ is finite by \cite[Corollary A.13]{jiangBoundednessKltGood2023}, the dimension of $G_s$ coincides with $\dim(G)$. Hence, the kernel of the multiplication map has maximal variation in all $s\in S$. By Theorem \ref{thm:viehweg ampleness}, there exist $b \gg a \gg 0$ such that $ \det(\cQ)^a \otimes \det(\mathcal{E})^b$
is ample on $S$.

% By Theorem \ref{thm:Viehweg package}.(3), $S^\alpha\cE\otimes\det(\cQ)^{-\beta}$ is weakly positive over $S$ for some  natural numbers $\alpha,\beta>0$. 
% It follows from \cite[Lemma 2.27]{viehwegQuasiprojectiveModuliPolarized1995} that 
% $$\det(\cE)=\det \big(f_*(\cO_X(qlK_{X/S})\otimes L^q)\big)\otimes\det(f_*L)^{-\frac{qr'}{r}}$$
% is ample on $S$. Moreover, 
By \cite[Proposition 7.9 and Lemma 9.26]{viehwegQuasiprojectiveModuliPolarized1995}, some power of $\det(\cQ)^a \otimes \det(\mathcal{E})^b$ descends to a line bundle $\Lambda$ on $G_{klt}(d, u, \sigma)$. Since $\widetilde{G}_{klt}(d,u,\sigma)$ is the quotient of $S$ by a finite group, the pullback $\delta^*\Lambda$ is ample on $\widetilde{G}_{klt}(d,u,\sigma)$. In particular, $\widetilde{G}_{klt}(d,u,\sigma)$ is a quasi-projective scheme.

If, moreover, the non-normal locus of $G_{klt}(d, u, \sigma)$ is proper, then by \cite[\uppercase\expandafter{\romannumeral3}, 2.6.2]{grothendieckElementsGeometrieAlgebrique1961}, the ampleness of $\delta^*\Lambda$ implies the ampleness of $\Lambda$. Consequently, $G_{klt}(d, u, \sigma)$ is a quasi-projective scheme.
\end{proof}

\begin{rem}
There exists an example of a line bundle $L$ on a non-normal non-proper space $X$ such that $L$ is not ample, whereas its pullback to the normalization $\widetilde{X}$ is ample
\cite[Proposition 2]{kollarTwoExamplesSurfaces2011}. Therefore, the quasi-projectivity of $\widetilde{G}_{klt}(d,u,\sigma)$ does not imply that of $G_{klt}(d,u,\sigma)$.
\end{rem}
        %%%%%%%%%%%%%%%%%%%%%%%%%%%%%%%%%%%%%%%%%%%%%%%
	\bibliographystyle{alphaurl}
		\bibliography{qproj}
		%%%%%%%%%%%%%%%%%%%%%%%%%%%%%%%%%%%%%%%%%%%%%%%%%%%
		
	\end{document}